\documentclass[11pt]{amsart}
\usepackage[latin1]{inputenc}
\usepackage{epsfig}
\usepackage{color}
\usepackage[british,english]{babel}
\usepackage{amsthm}
\usepackage{amsmath}
\usepackage{amsfonts}
\usepackage{amssymb}
\usepackage{graphicx}
\newtheorem{corollary}{Corollary}[section]
\newtheorem{definition}[corollary]{Definition}

\newtheorem{lemma}[corollary]{Lemma}
\newtheorem{proposition}[corollary]{Proposition}
\newtheorem{remark}[corollary]{Remark}
\newtheorem{theorem}[corollary]{Theorem}
\newfont{\sBlackboard}{msbm10 scaled 900}

\newcommand{\mylabel}[1]{\label{#1}
            \ifx\undefined\stillediting
            \else \fbox{$#1$}\fi }
\newcommand{\BE}{\begin{equation}}

\newcommand{\EEQ}{\end{equation}}
\newcommand{\rfb}[1]{\mbox{\rm
   (\ref{#1})}\ifx\undefined\stillediting\else:\fbox{$#1$}\fi}

\newfont{\Blackboard}{msbm10 scaled 1200}

\newfont{\roma}{cmr10 scaled 1200}

\def\CC{\rm \hbox{C\kern-.56em\raise.4ex
         \hbox{$\scriptscriptstyle |$}\kern+0.5 em }}
\newcommand{\be}{\begin{equation}}
\newcommand{\ee}{\end{equation}}
\newcommand{\beq}{\begin{eqnarray}}
\newcommand{\eeq}{\end{eqnarray}}
\newcommand{\beqs}{\begin{eqnarray*}}
\newcommand{\eeqs}{\end{eqnarray*}}
\newcommand{\bt}{\begin{Theorem}}
\newcommand{\et}{\end{Theorem}}
\newcommand{\br}{\begin{remark}}
\newcommand{\er}{\end{remark}}
\newcommand{\bc}{\begin{Corollary}}
\newcommand{\ec}{\end{Corollary}}
\newcommand{\el}{\end{Lemma}}
\newcommand{\bd}{\begin{definition}}
\newcommand{\ed}{\end{definition}}

\newcommand{\mm}    {{\hbox{\hskip 0.5pt}}}

\newcommand{\bluff} {{\hbox{\raise 15pt \hbox{\mm}}}}
\makeatletter
\def\section{\@startsection {section}{1}{\z@}{-3.5ex plus -1ex minus
    -.2ex}{2.3ex plus .2ex}{\large\bf}}
\makeatother
\def\be{\begin{equation}}
\def\ee{\end{equation}}

\begin{document}

\thispagestyle{empty}
\title[one dimensional model of  fluid-structure interaction]{Feedback stabilization of a simplified model of  fluid-structure interaction on a tree}
\author{Ka\"{\i}s AMMARI}
\address{UR Analysis and Control of Pde, UR 13ES64, Department of Mathematics, Faculty of Sciences of Monastir, University of Monastir, Tunisia}
\email{kais.ammari@fsm.rnu.tn}
\author{Farhat Shel}
\address{UR Analysis and Control of Pde, UR 13ES64, Department of Mathematics, Faculty of Sciences of Monastir, University of Monastir, Tunisia}
\email{farhat.shel@ipeit.rnu.tn}

\author{Muthusamy VANNINATHAN} 
\address{TIFR-CAM, Post Bag 6503, Bangalore 560065, India} \email{vanni@math.tifrbng.res.in}

\begin{abstract}
In this paper we study the dynamic feedback stability for a simplified model of fluid-structure interaction on a tree. 
We prove that, under some conditions, the energy of the solutions of the system decay exponentially to zero when the time tends to infinity. 
Our technique is based on a frequency domain method and a special 
analysis for the resolvent. 
\end{abstract}

\maketitle

\tableofcontents

\vfill\break
\section{Introduction}

\setcounter{equation}{0}

First of all, we introduce some notations needed to formulate the problem
under consideration, as introduced in \cite{Abd12} or in \cite{Mer08}. Let $%
\mathcal{T}$ be a tree (i.e. a planar connected graph without closed paths).
By the multiplicity of a vertex of $\mathcal{T}$ we mean the number of
edges that branch out from the vertex. If the multiplicity is equal to one,
the vertex is called exterior; otherwise, it is said to be interior. We
denote by $\{e_{1},...,e_{N}\}$ the set of edges of $\mathcal{T}$ and $%
\{a_{1},...,a_{N+1}\}$ its set of vertices and we assume that $a_{1}$ is the
root of $\mathcal{T}$ which will be denoted by $\mathcal{R}$, that $e_{1}$
is the edge containing $\mathcal{R}$ and $a_{2}$ is its vertex different
from $\mathcal{R}.$

We denote by $\mathcal{M}$ the set of the interior vertices of $\mathcal{T}$
and by $\mathcal{S}$ the set of the exterior vertices, except $\mathcal{R}$
and denote $I_{\mathcal{M}}$ and $I_{\mathcal{S}}$ the sets of indices of
interior and exterior vertices, except $\mathcal{R},$ respectively. Then $%
I=I_{\mathcal{M}}\cup I_{\mathcal{S}}$ is the set of indices of all
vertices, except the root $\mathcal{R}.$ We denote by $J$ the set $%
\{1,...,N\}$ and for $k\in I$ we will denote by $J_{k}$ the set of indices
of edges adjacent to $a_{k}$. If $k\in \mathcal{M},$ then the indice of the
unique element of $J_{k}$ will be denoted by $j_{k}.$ $.$

Furthermore, the length of the edge $e_{j}$ will be denoted by $\ell _{j}.$
Then $e_{j}$ may be parametrized by its arc length by means of the function $%
\pi _{j}:[0,\ell _{j}]\longrightarrow e_{j},\;x_{j}\mapsto \pi _{j}(x_{j}),$
and sometimes identified with the interval $(0,\ell _{j}).$

\noindent For a function $y:\mathcal{T}\longrightarrow \mathbb{C}$ we set $%
y^{j}=y\circ \pi _{j}$ its restriction to the edge $e_{j}$ and we will
denote $y^{j}(x)=y^{j}(\pi _{j}(x))$ for any $x$ in $(0,\ell _{j}).$

The incidence matrix $D=(d_{kj})_{(N+1)\times N}$ is defined by 
\begin{equation*}
d_{kj}=\left\{ 
\begin{tabular}{l}
$1$ if $\pi _{j}(\ell _{j})=a_{k},$ \\ 
$-1$ if $\pi _{j}(0)=a_{k},$ \\ 
$0$ otherwise.
\end{tabular}
\right.
\end{equation*}

We denote by $\left\langle .,.\right\rangle $ and $\left\| .\right\| $ the
inner product and norm in $L^{2}$-space, respectively.

\begin{figure}[th]
\centering
\includegraphics[width=12cm, keepaspectratio =true]{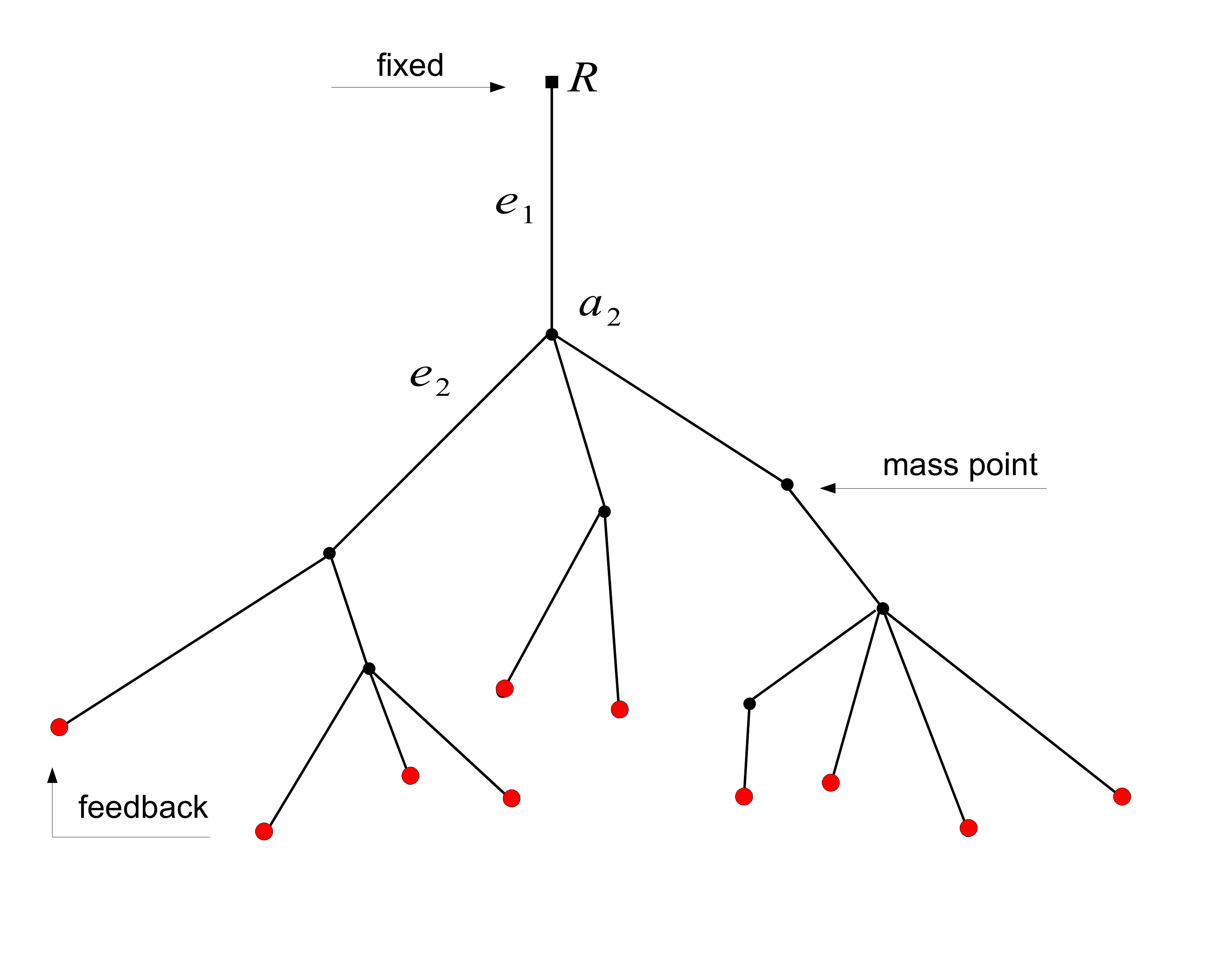}
\caption{A tree-shaped network }
\label{figure1}
\end{figure}

In this paper, we study the stability of a model of fluid propagating in a $%
1 $-$d$ network, under some feedback forces applied at exteriors nodes, with the
presence of point mass at inner nodes, see Figure \ref{figure1}. At rest,
the network coincides with the tree $\mathcal{T}$.

More precisely, we consider the following initial and boundary value
problem. 
\begin{equation}
\left\{ 
\begin{tabular}{l}
$y_{tt}^{j}-y_{xx}^{j}=0$ in $(0,\ell _{j})\times (0,\infty ),\;\;j\in J$,
\\ 
\\ 
$\sum\limits_{j\in J_{k}}d_{kj}y_{x}^{j}(a_{k},t)=s_{k}^{\prime
}(t),\;\;\;\;\;k\in I_{\mathcal{M}},$ \\ 
$s_{k}^{\prime \prime }(t)+s_{k}(t)=-y_{t}(a_{k},t),\;\;\;\;\;k\in I_{%
\mathcal{M}},$ \\ 
$y^{j}(a_{k},t)=y^{l}(a_{k},t),\;\;\;\;j,l\in J_{k},\;\;k\in I_{\mathcal{M}%
}, $ \\ 
$y^{1}(a_{1},t)=0,$ \\ 
$d_{kj_{k}}y_{x}^{j_{k}}(a_{k},t)=-y_{t}(a_{k},t),\;\;\;\;\;k\in I_{\mathcal{%
S}},$ \\ 
\\ 
$s_{k}(0)=s_{k,0},\;\;s_{k}^{\prime }(0)=s_{k,1},\;\;\;\;\;k\in I_{\mathcal{M%
}},$ \\ 
$y^{j}(x,0)=y_{0}^{j}(x),$ $\;y_{t}^{j}(x,0)=y_{k}^{1}(x),$ \ $x\in (0,\ell
_{j}),\;j\in J,$%
\end{tabular}
\right.  \label{1.1}
\end{equation}
where $y^{j}:[0,\ell _{j}]\times (0,\infty )\longrightarrow \mathbb{R},j\in
J $ represents the velocity potential of the fluid on the edge $e_{j}$ and$%
\;s_{k}:(0,\infty )\longrightarrow \mathbb{R},\;k\in I_{\mathcal{M}}$
denotes the displacement of the point mass occupying the node $a_{k}.$ These functions allow us to identify the network with its rest graph. See \cite
{Zua06} for more details.

This simplified model of fluid-structure interaction draws on the work of
Ervedoza and Vanninathan \cite{Erv14}; they consider a fluid occupying a
domain in two dimensions and a solid immersed in it and prove some results
of controllability of such system; see also \cite{Tuc09}. The problem of
fluid structure interaction in one dimension has been studied by several
authors. In \cite{Zua03} the authors study the asymptotic behavior of a one
dimensional model of mass point moving in a fluid. They consider the same
system in \cite{Zua06'} but with a finite number of mass points floating in
the fluid. Recently Tucsnak et al. \cite{Tuc13} studied the controllability
of a similar system.

Note that the point-wise (or boundary) stabilization on the wave equation
has been treated during the last few years, see for example \label{Amm01d}\cite
{Amm01d} for one string, and \cite{Amm04, Amm05, Amm15} for some networks of
strings.

The main result of this paper asserts that, under some conditions, the energy
of the solutions of the dissipative system decay exponentially to zero when
the time tends to infinity. The method is based on a frequency domain method
and a special analysis for the resolvent.

If $(y,s)=((y^{j})_{j\in J},(s_{k})_{k\in \mathcal{M}})$ is a solution of (%
\ref{1.1}) we define the energy of $(y,s)$ at instant $t$ by 
\begin{equation*}
E(t)=\frac{1}{2}\sum\limits_{j\in J}\int_{0}^{\ell _{j}}\left( \left|
y_{t}^{j}(x,t)\right| ^{2}+\left| y_{x}^{j}(x,t)\right| ^{2}\right) dx+\frac{%
1}{2}\sum\limits_{k\in I_{\mathcal{M}}}\left( \left| s_{k}^{\prime
}(t)\right| ^{2}+\left| s_{k}(t)\right| ^{2}\right) .
\end{equation*}
Simple formal calculations show that a sufficiently smooth solution of (\ref
{1.1}) satisfies the energy estimate 
\begin{equation}
E(0)-E(t)=\sum\limits_{k\in I_{\mathcal{S}}}\int_{0}^{t}\left|
y_{t}^{j_{k}}(a_{k},s)\right| ^{2}ds,\;\;\;\forall t\geq 0.  \label{1.3}
\end{equation}
\noindent In particular (\ref{1.3}) implies that 
\begin{equation*}
E(t)\leq E(0),\;\;\;\forall t\geq 0.
\end{equation*}
Estimate above suggests that the natural wellposedness space for (\ref{1.1})
is 
\begin{equation*}
H=V\times \prod_{j\in J}L^{2}(0,\ell _{j})\times \left( \prod_{k\in I_{%
\mathcal{M}}}\mathbb{C}\right)^{2}
\end{equation*}
where 
\begin{equation*}
V=\left\{ \Phi \in \prod_{j\in J}H^{1}(0,\ell _{j}),\;\Phi ^{1}(\mathcal{R}%
)=0,\;\Phi ^{j}(a_{k})=\Phi ^{l}(a_{k}),\;j,l\in J_{k},\;k\in I_{\mathcal{M}%
}\right\} .
\end{equation*}

We can rewrite the system (\ref{1.1}) as a first order differential
equation, by putting $z(t)=\left( 
\begin{array}{c}
y(t) \\ 
y^{\prime }(t) \\ 
s(t) \\ 
s^{\prime }(t)
\end{array}
\right) :$%
\begin{equation*}
z^{\prime }(t)=\mathcal{A}z(t),\;\;z(0)=z_{0}=\left( 
\begin{array}{c}
y_{0} \\ 
y_{1} \\ 
s^{0} \\ 
s^{1}
\end{array}
\right) ,
\end{equation*}
where 
\begin{equation*}
\mathcal{A}\left( 
\begin{array}{c}
y \\ 
v \\ 
p \\ 
q
\end{array}
\right) =\left( 
\begin{array}{c}
v \\ 
\dfrac{d^{2}y}{dx^{2}} \\ 
q \\ 
-p-v_{\mathcal{M}}
\end{array}
\right) ,\forall \left( 
\begin{array}{c}
y \\ 
v \\ 
p \\ 
q
\end{array}
\right) \in \mathcal{D}(\mathcal{A}),
\end{equation*}
with $v_{\mathcal{M}}=(v(a_{k}))_{k\in I_{\mathcal{M}}},$ and 
\begin{equation*}
\mathcal{D}(\mathcal{A})=\left\{ (y,v,p,q)\in \left[ \prod_{j\in
J}H^{2}(0,\ell _{j})\cap V\right] \times V\times \left( \prod_{k\in I_{%
\mathcal{M}}}\mathbb{C}\right)^{2} \text{ satisfying (\ref{1.33}) }\right\}
\end{equation*}
where 
\begin{equation}
\left\{ 
\begin{tabular}{l}
$\sum_{j\in J_{k}}d_{kj}\dfrac{dy^{j}}{dx}(a_{k})=q_{k},\;\forall k\in I_{%
\mathcal{M}},$ \\ 
$d_{kj}\dfrac{dy^{j_{k}}}{dx}(a_{k})=-v^{j_{k}}(a_{k}),\;\forall k\in I_{%
\mathcal{S}}.$%
\end{tabular}
\right.  \label{1.33}
\end{equation}
To simplify the notations, sometimes, we will write $y(a_{k})$ instead of $%
y^{j}(a_{k})$ for $y$ in $V$.

The outline of this work is the following. In Section \ref{sec2} we prove the
existence and uniqueness of solutions for system (\ref{1.1}). Section \ref{sec3} is devoted to prove the exponential stability of the associated semi group.
Finally, in Section \ref{sec4}, we prove the lack of exponential stability if the
graph contain a circuit or if there is at least one uncontrolled exterior
node (other than the root). The section \ref{sec5} is devoted to the study of the case of a chain with non equal mass points.

\section{Wellposedness} \label{sec2}

\begin{lemma}
The operator $\mathcal{A}$ generates a $\mathcal{C}_{0}$-semigroup of
contractions $(S(t))_{t\in \mathbb{R}^{+}}.$
\end{lemma}

\begin{proof}
It is clear that the operator $\mathcal{A}$ is dissipative, moreover, for
every $z=(y,v,p,q)\in \mathcal{D}(\mathcal{A}),$%
\begin{equation}
Re\left( \left\langle \mathcal{A}z,z\right\rangle _{H}\right) =-\sum_{k\in
I_{\mathcal{S}}}\left| v^{j_{k}}(a_{k})\right| ^{2}\leq 0.  \label{2.71}
\end{equation}

Now we prove that every positive real number $\lambda $ belongs to $\rho (%
\mathcal{A})$, the resolvent set of $\mathcal{A}.$ For this, let $%
Z=(f,g,c,d)\in H$ and we solve the equation 
\begin{equation}
(\lambda -\mathcal{A})z=Z  \label{1.2}
\end{equation}
with $z$ in $\mathcal{D}(\mathcal{A}).$

\noindent We rewrite (\ref{1.2}) explicitly 
\begin{equation}
\left\{ 
\begin{tabular}{l}
$\lambda y^{j}-v^{j}=f^{j},\;\;j\in J,$ \\ 
$\lambda v^{j}-\dfrac{d^{2}y^{j}}{dx^{2}}=g^{j},\;\;j\in J,$ \\ 
$\lambda p_{k}-q_{k}=c_{k},\;\;k\in I_{\mathcal{M}},$ \\ 
$\lambda q_{k}+p_{k}+v_{k}(a_{k})=d_{k},\;\;k\in I_{\mathcal{M}}.$%
\end{tabular}
\right.  \label{1.22}
\end{equation}
We eliminate $(v,q)$ in (\ref{1.22}) to get 
\begin{eqnarray}
\lambda ^{2}y^{j}-\dfrac{d^{2}y^{j}}{dx^{2}} &=&g^{j}+\lambda f^{j},\;\;j\in
J,  \label{1.31} \\
(\lambda ^{2}+1)p_{k}+v_{k}(a_{k}) &=&d_{k}+\lambda c_{k},\;\;k\in I_{%
\mathcal{M}}.  \label{1.4}
\end{eqnarray}
Let $w$ in $V.$ Multiplying (\ref{1.31}) by $w^{j}$ in $L^{2}(0,\ell _{j})$
and summing over $j\in J,$%
\begin{eqnarray}
&&\lambda ^{2}\sum_{j\in J}\left( \int_{0}^{\ell _{j}}y^{j}\overline{w^{j}}%
dx+\int_{0}^{\ell _{j}}\dfrac{dy^{j}}{dx}\dfrac{d\overline{w^{j}}}{dx}%
dx\right) -\sum_{k\in I_{\mathcal{M}}}\overline{w(a_{k})}q_{k}+\sum_{k\in I_{%
\mathcal{S}}}\overline{w(a_{k})}(a_{k})  \notag \\
&=&\sum_{j\in J}\int_{0}^{\ell _{j}}(g^{j}+\lambda f^{j})\overline{w^{j}}dx.
\label{1.222}
\end{eqnarray}
Multiplying (\ref{1.4}) by $r_{k}\in \mathbb{C}$ and summing over $k\in I_{%
\mathcal{M}},$ we get 
\begin{equation}
(\lambda ^{2}+1)\sum_{k\in I_{\mathcal{M}}}p_{k}\overline{r_{k}}+\sum_{k\in
I_{\mathcal{M}}}v(a_{k})\overline{r_{k}}=\sum_{k\in I_{\mathcal{M}%
}}(d_{k}+\lambda c_{k})\overline{r_{k}}.  \label{1.333}
\end{equation}
Summing (\ref{1.222}) and (\ref{1.333}) we get 
\begin{multline*}
\lambda ^{2}\sum_{j\in J}\left( \left\langle y^{j},w^{j}\right\rangle
+\left\langle \dfrac{dy^{j}}{dx},\dfrac{dw^{j}}{dx}\right\rangle \right)
+(\lambda ^{2}+1)\sum_{k\in I_{\mathcal{M}}}p_{k}\overline{r_{k}}+\sum_{k\in
I_{\mathcal{M}}}\left( v(a_{k})\overline{r_{k}}-\overline{w(a_{k})}%
q_{k}\right) \\
+\sum_{k\in I_{\mathcal{S}}}\overline{w(a_{k})}v(a_{k})=\sum_{j\in
J}\left\langle g^{j}+\lambda f^{j},w^{j}\right\rangle +\sum_{k\in I_{%
\mathcal{M}}}(d_{k}+\lambda c_{k})\overline{r_{k}}
\end{multline*}
to obtain 
\begin{equation*}
a((y,p),(w,r))=f(w,r)
\end{equation*}
where 
\begin{eqnarray*}
a((y,p),(w,r)) &=&\lambda ^{2}\sum_{j\in J}\left( \left\langle
y^{j},w^{j}\right\rangle +\left\langle \dfrac{dy^{j}}{dx},\dfrac{dw^{j}}{dx}%
\right\rangle \right) +(\lambda ^{2}+1)\sum_{k\in I_{\mathcal{M}}}p_{k}%
\overline{r_{k}} \\
&&+\lambda \sum_{k\in I_{\mathcal{M}}}\left( y(a_{k})\overline{r_{k}}-%
\overline{w(a_{k})}p_{k}\right) +\lambda \sum_{k\in I_{\mathcal{S}}}%
\overline{w(a_{k})}y(a_{k})
\end{eqnarray*}
and 
\begin{eqnarray*}
f(w,r) &=&\sum_{j\in J}\left\langle g^{j}+\lambda f^{j},w^{j}\right\rangle
+\sum_{k\in I_{\mathcal{M}}}(d_{k}+\lambda c_{k})\overline{r_{k}}-\sum_{k\in
I_{\mathcal{M}}}\overline{w(a_{k})}c_{k} \\
&&+\sum_{k\in I_{\mathcal{M}}}f(a_{k})\overline{r_{k}}+\sum_{k\in I_{%
\mathcal{S}}}\overline{w(a_{k})}f(a_{k}).
\end{eqnarray*}

$a$ is a continuous sesquilinear form on $V\times \prod\limits_{k\in I_{%
\mathcal{M}}}\mathbb{C}$ and $f$ is a continuous anti-linear form on $%
V\times \prod\limits_{k\in I_{\mathcal{M}}}\mathbb{C}.$ Moreover 
\begin{equation*}
a((w,r),(w,r))=\lambda ^{2}\sum_{j\in J}\left( \left\| w^{j}\right\|
^{2}+\left\| \dfrac{dw^{j}}{dx}\right\| ^{2}\right) +(\lambda
^{2}+1)\sum_{k\in I_{\mathcal{M}}}\left| r_{k}\right| ^{2}-2\mathbf{i}%
Im\left( \sum_{k\in I_{\mathcal{M}}}\overline{w(a_{k})}r_{k}\right)
\end{equation*}
and then 
\begin{equation*}
\left| a((w,r),(w,r))\right| \geq \lambda ^{2}\left[ \sum_{j\in J}\left(
\left\| w^{j}\right\| ^{2}+\left\| \dfrac{dw^{j}}{dx}\right\| ^{2}\right)
+\sum_{k\in I_{\mathcal{M}}}\left| r_{k}\right| ^{2}\right]
\end{equation*}
that is, $a$ is coercive. The conclusion result immediately from the Lax
Milgram lemma.
\end{proof}

\begin{proposition}
Suppose that $(y_{0},y_{1},s^{0},s^{1})\in H.$ Then the problem (\ref{1.1})
admits a unique solution 
\begin{equation*}
(y,y^{\prime },s,s^{\prime })\in \mathcal{C}([0,+\infty );H).
\end{equation*}
If $(y_{0},y_{1},s^{0},s^{1})\in \mathcal{D}(\mathcal{A})$ then 
\begin{equation*}
(y,y^{\prime },s,s^{\prime })\in \mathcal{C}([0,+\infty ),\mathcal{D}(%
\mathcal{A}))\cap \mathcal{C}^{1}([0,+\infty );\mathcal{H}).
\end{equation*}
Moreover $(y,s)$ satisfies the energy estimate (\ref{1.3}).
\end{proposition}

\section{Exponential stability} \label{sec3}

It is clear that if $\mathcal{T}$ contains an edge $e_{j}$, not attached to a leaf, with length $\ell _{j}\in \pi \mathbb{N}$ then $\mathbf{i}$ is
eigenvalue of $\mathcal{A}$ with eigenvector $z=(y,v,p,q)$ such that $y^{j}= \mathbf{i}\sin x,$ $v^{j}=-\sin x,$ $p_{k}=1$ and $q_{k}=\mathbf{i}$ when $
a_{k}$ is the nearest end of $e_{j}$ to the root $\mathcal{R},$ and all the
other compoents of $z$ are null.

In the following, the tree $\mathcal{T}$ is said to be a \textbf{Pi-tree} if
it has has no edges of length $m\pi $ with $m\in \mathbb{N}^{\ast },$ except
maybe those attached to leaves. Then we have the following result:

\begin{lemma}
\label{lem2} The spectrum of $\mathcal{A}$ contains no point on the
imaginary axis if and only if $\mathcal{T}$ is a Pi-tree.
\end{lemma}

\begin{proof}
Since $\mathcal{A}$ has compact resolvent, its spectrum $\sigma (\mathcal{A}%
) $ only consists of eigenvalues of $\mathcal{A}.$ We will show that the
equation 
\begin{equation}
\mathcal{A}z=\mathbf{i}\beta z  \label{2.6}
\end{equation}
with $z=\left( 
\begin{array}{c}
y \\ 
v \\ 
p \\ 
q
\end{array}
\right) \in \mathcal{D}(\mathcal{A})$ and $\beta \in \mathbb{R}$ has only
trivial solution.

By taking the inner product of (\ref{2.6}) with $z\in H$ and using that 
\begin{equation*}
Re\left( \left\langle \mathcal{A}z,z\right\rangle _{H}\right) =-\sum_{k\in
I_{\mathcal{S}}}\left| v(a_{k})\right| ^{2}
\end{equation*}
we obtain that $v(a_{k})=0$ for $k\in I_{\mathcal{S}}$ and then, 
\begin{equation}
\dfrac{dy}{dx}(a_{k})=0\text{ for }k\in I_{\mathcal{S}}.  \label{cc}
\end{equation}
Now the equation (\ref{2.6}) can be rewritten explicitly as 
\begin{eqnarray}
v^{j} &=&\mathbf{i}\beta y^{j},\;\;j\in J,  \label{s1} \\
\dfrac{d^{2}y^{j}}{dx^{2}} &=&\mathbf{i}\beta v^{j},\;\;j\in J,  \label{s2}
\\
q_{k} &=&\mathbf{i}\beta p_{k},\;\;k\in I_{\mathcal{M}},  \label{s3} \\
-p_{k}-v(a_{k}) &=&\mathbf{i}\beta q_{\kappa },\;\;k\in I_{\mathcal{M}}.
\label{s4}
\end{eqnarray}

\noindent If $\beta =0$ then $v=0,$ $q=0$ and $p=0.$

\noindent Multiplying the second equation in the above system by $y^{j}$ and
then summing over $j,$ we obtain 
\begin{equation*}
\sum_{j\in J}\left\| \dfrac{dy^{j}}{dx}\right\| ^{2}=0,
\end{equation*}
which implies, using continuity condition of $y$ at inner nodes and its
Dirichlet condition at $\mathcal{R},$ that $y=0.$

Next, we suppose that $\beta \neq 0.$ We have, using (\ref{s1}-\ref{s4}), 
\begin{eqnarray}
\beta ^{2}y^{j}+\dfrac{d^{2}y^{j}}{dx^{2}} &=&0,\;\;j\in J,  \label{eq} \\
(\beta ^{2}-1)p_{k} &=&v(a_{k}),\;\;k\in I_{\mathcal{M}}.  \label{eq2}
\end{eqnarray}
The function $y^{j},$ $j\in J$ is then of the form 
\begin{equation*}
y^{j}=\alpha _{1}\sin (\beta x)+\alpha _{2}\cos (\beta x).
\end{equation*}
Using (\ref{cc}) and (\ref{s1}) we deduce that, for $k\in I_{\mathcal{S}},$ $%
y^{j_{k}}=0$ and then $v^{j_{k}}=0.$

For the sequel of the proof we consider two cases:

\textit{First case:} $\beta ^{2}=1$. From (\ref{eq2}), (\ref{s1}) and the
Dirichlet condition at $\mathcal{R},$ we deduce that $y^{j}(0)=y^{j}(\ell
_{j})=0$ for $j\in J.$ Since $\mathcal{T}$ is a Pi-tree we conclude that $%
y^{j}=0$ for $j\in J.$ Return back to the balance conditions and (\ref{s3}),
one can deduce that $q=p=0$ and hence $y=0.$

\textit{Second case: }$\beta ^{2}\neq 1.$ Let $a_{k}$ the second end of an
edge $e_{j}$ attached to a leaf. Then $p_{k}=0$ and $q_{k}=0.$ Let $e_{l}$
the edge ended by $a_{k}$ and not attached to a leaf. We have $%
y^{l}(a_{k})=0 $ and $\dfrac{dy^{l}}{dx}(a_{k})=0.$ Then by (\ref{eq}) $%
y^{l}=0.$ We iterate such procedure from root to leaves to conclude that $%
y=0.$
\end{proof}

The main result of this paper concerns the precise asymptotic behavior of
the solution of (\ref{1.1}). Our technique is based on a frequency domain
method and a special analysis for the resolvent.

Recall that the system (\ref{1.1}) is said to be exponentially stable if
there exist two constants $M,\omega >0,$ such that for all $%
(y_{0},y_{1},s^{0},s^{1})\in H,$ 
\begin{equation*}
E(t)\leq Me^{-\omega t}\left\| (y_{0},y_{1},s^{0},s^{1})\right\|
_{H}^{2},\;\;\forall t\geq 0.
\end{equation*}
Then, our main result is the following:

\begin{theorem}
\label{th32} The system defined by equations (\ref{1.1}) is exponentially
stable if and only if $\mathcal{T}$ is a Pi-tree.
\end{theorem}

\begin{proof}
By classical result (see Huang \cite{Hua85} and Pr\"{u}ss \cite{Pru84}) it
suffices to show that the corresponding operator $\mathcal{A}$ satisfies the
following two conditions: 
\begin{equation}
\rho (\mathcal{A})\supset \{\mathbf{i}\beta \mid \beta \in \mathbb{R}%
\}\equiv \mathbf{i}\mathbb{R},  \label{2.9}
\end{equation}
and 
\begin{equation}
\lim \sup\limits_{\left| \beta \right| \rightarrow \infty }\left\| (\mathbf{i%
}\beta -\mathcal{A})^{-1}\right\| <\infty ,  \label{2.10}
\end{equation}
where $\rho (\mathcal{A})$ denotes the resolvent set of the operator $%
\mathcal{A}.$

By lemma \ref{lem2} the condition (\ref{2.9}) is satisfied. Suppose that the
condition (\ref{2.10}) is false. By the Banach-Steinhaus Theorem (see \cite
{Bre83}), there exist a sequence of real numbers $\beta _{n}\rightarrow
\infty $ ($\beta _{n}>0$ without loss of generality) and a sequence of
vector $z_{n}=\left( 
\begin{array}{c}
y_{n} \\ 
v_{n} \\ 
p_{n} \\ 
q_{n}
\end{array}
\right) \in \mathcal{D}(\mathcal{A})$ with $\left\| z_{n}\right\| _{H}=1$
such that 
\begin{equation}
\left\| (\mathbf{i}\beta _{n}I-\mathcal{A})z_{n}\right\| _{H}\longrightarrow
0\text{\ \ \ as\ \ \ }n\longrightarrow \infty ,  \label{2.11}
\end{equation}
i.e. 
\begin{eqnarray}
\mathbf{i}\beta _{n}y_{n}^{j}-v_{n}^{j} &\equiv &f_{n}^{j}\longrightarrow 0%
\text{\ \ \ \ \ in \ \ }H^{2}(0,\ell _{j}),  \label{2.12} \\
\mathbf{i}\beta _{n}v_{n}^{j}-\dfrac{d^{2}y_{n}^{j}}{dx^{2}} &\equiv
&g_{n}^{j}\longrightarrow 0\text{\ \ \ \ in \ \ }L^{2}(0,\ell _{j}),
\label{2.13} \\
\mathbf{i}\beta _{n}p_{k,n}-q_{k,n} &\equiv &h_{k,n}\longrightarrow 0\text{\
\ \ in\ \ }\mathbb{C},  \label{2.14} \\
\mathbf{i}\beta _{n}q_{k,n}+p_{k,n}+v_{n}(a_{k}) &\equiv
&r_{k,n}\longrightarrow 0\text{\ \ \ in\ \ }\mathbb{C}.  \label{2.15}
\end{eqnarray}
Our goal is to derive from (\ref{2.11}) that $\left\| z_{n}\right\| _{H}$
converge to zero, thus, a contradiction.

The proof is divided in three steps:

\textit{First step.} Recall that for every $j$ in $J,$%
\begin{equation*}
\frac{\left\| v^{j}\right\| _{\infty }}{\beta _{n}^{1/2}}\leq \frac{C_{1}}{%
\beta _{n}^{1/2}}\left\| \dfrac{dv^{j}}{dx}\right\| ^{1/2}\left\|
v^{j}\right\| ^{1/2}+C_{2}\frac{\left\| v^{j}\right\| }{\beta _{n}^{1/2}},
\end{equation*}
for some positives constants $C_{1}$ and $C_{2}.$ This implies, using (\ref
{2.12}), that $\frac{\left\| v^{j}\right\| _{\infty }}{\beta _{n}^{1/2}}$ is
bounded.\newline
Then, for every $k$ in $I_{\mathcal{M}},$ by (\ref{2.15}), $q_{k,n}$
converge to zero and then $\beta _{n}p_{k,n}$ converge to zero in view of (%
\ref{2.14}). In particular $p_{k,n}$ converge to zero.

\textit{Second step.} We notice that from (\ref{2.71}) we have 
\begin{equation}
\left\| (\mathbf{i}\beta _{n}I-\mathcal{A})z_{n}\right\| _{H}\geq \left|
Re\left\langle \mathbf{i}\beta _{n}I-\mathcal{A})z_{n},z_{n}\right\rangle
_{H}\right| =\sum_{k\in I_{\mathcal{S}}}\left| v_{n}(a_{k})\right| ^{2}.
\label{2.16}
\end{equation}
Then, by (\ref{2.11}) 
\begin{equation*}
\left| v_{n}(a_{k})\right| \longrightarrow 0,\;\;\forall k\in I_{\mathcal{S}%
}.
\end{equation*}
This further leads to 
\begin{equation}
\left| \beta _{n}y_{n}(a_{k})\right| \longrightarrow 0,\;\;\forall k\in I_{%
\mathcal{S}}.  \label{2.18}
\end{equation}
due to (\ref{2.12}) and the trace theorem.

\noindent We have also 
\begin{equation}
\dfrac{dy^{j_{k}}}{dx}(a_{k})\longrightarrow 0,\;\;\forall k\in I_{\mathcal{S%
}}.  \label{2.19}
\end{equation}

\textit{Third step.} Substituting (\ref{2.12}) into (\ref{2.13}) and (\ref
{2.14}) into (\ref{2.15}) to get 
\begin{equation}
-\beta _{n}^{2}y_{n}^{j}-\dfrac{d^{2}y_{n}^{j}}{dx^{2}}=g_{n}^{j}+\mathbf{i}%
\beta _{n}f_{n}^{j},\;\;j\in J  \label{2.20}
\end{equation}
\begin{equation}
-\beta _{n}^{2}p_{k,n}+v_{n}(a_{k})=r_{k,n}+\mathbf{i}\beta _{n}h_{k,n}.
\label{2.21}
\end{equation}
Next, we take the inner product of (\ref{2.20}) with $b\dfrac{dy_{n}^{j}}{dx}
$ in $L^{2}(0,\ell _{j})$ for $b\in \mathcal{C}^{1}([0,\ell _{j}])$ we get 
\begin{multline*}
\frac{1}{2}\beta _{n}^{2}\left[ \left| y_{n}^{j}(x)\right| ^{2}b(x)\right]
_{0}^{\ell _{j}}+\frac{1}{2}\left[ \left| \dfrac{dy_{n}^{j}}{dx}(x)\right|
^{2}b(x)\right] _{0}^{\ell _{j}}+\left[ Re\left( \mathbf{i}\beta
_{n}f_{n}^{j}(x)\overline{y_{n}^{j}}(x)b(x)\right) \right] _{0}^{\ell _{j}}
\\
-\frac{1}{2}\int_{0}^{\ell _{j}}\left( \left| \beta _{n}y_{n}^{j}\right|
^{2}+\left| \dfrac{dy_{n}^{j}}{dx}\right| ^{2}\right) \dfrac{db}{dx}%
dx\longrightarrow 0.
\end{multline*}
With (\ref{2.18}) and (\ref{2.19}), this leads to 
\begin{equation*}
\frac{1}{2}\int_{0}^{\ell _{j_{k}}}\left( \left| \beta
_{n}y_{n}^{j_{k}}\right| ^{2}+\left| \dfrac{dy_{n}^{j_{k}}}{dx}\right|
^{2}\right) dx\longrightarrow 0
\end{equation*}
for every $k$ in $I_{\mathcal{S}},$ by taking $b=x$ or $b=\ell _{j_{k}}-x.$
Moreover, as in \cite{Far13} it follows that 
\begin{equation*}
\beta _{n}y_{n}(a_{s})\longrightarrow 0,\;\;\dfrac{dy_{n}^{j_{k}}}{dx}%
(a_{s})\longrightarrow 0,\;\;\text{and \ }Re\left( i\beta _{n}f_{n}(a_{s})%
\overline{y_{n}}(a_{s})\right) \longrightarrow 0
\end{equation*}
where $a_{s}$ is the end of $e_{j_{k}},$ different from $a_{k}.$ We then
conclude by iteration, that for every $j$ in $J,$ 
\begin{equation*}
\int_{0}^{\ell _{j}}\left( \left| \beta _{n}y_{n}^{j}\right| ^{2}+\left|
\partial _{x}y_{n}^{j}\right| ^{2}\right) dx\longrightarrow 0.
\end{equation*}

Finally, in view of (\ref{2.12}), we also get 
\begin{equation*}
\left\| v_{n}^{j}\right\| \longrightarrow 0,\;\;\text{for }j\in J.
\end{equation*}
which implies that $\left\| z_{n}\right\| _{H}\longrightarrow 0:$ clearly
contradicts (\ref{2.11}).
\end{proof}


\section{Some other cases} \label{sec4}

In this part we consider two particular cases. In the first, (Figure \ref
{figure2}), there is a circuit in the graph. We prove that, even with much
more controls, the exponential stability fails. The second case (Figure \ref
{figure3}) proves that when we eliminate a control of a leaf in the initial
case then the exponential stability fails. 
\begin{figure}[tbp]
\centering
\begin{minipage}[t]{7cm}
\centering
\includegraphics[width=8cm, keepaspectratio =true]{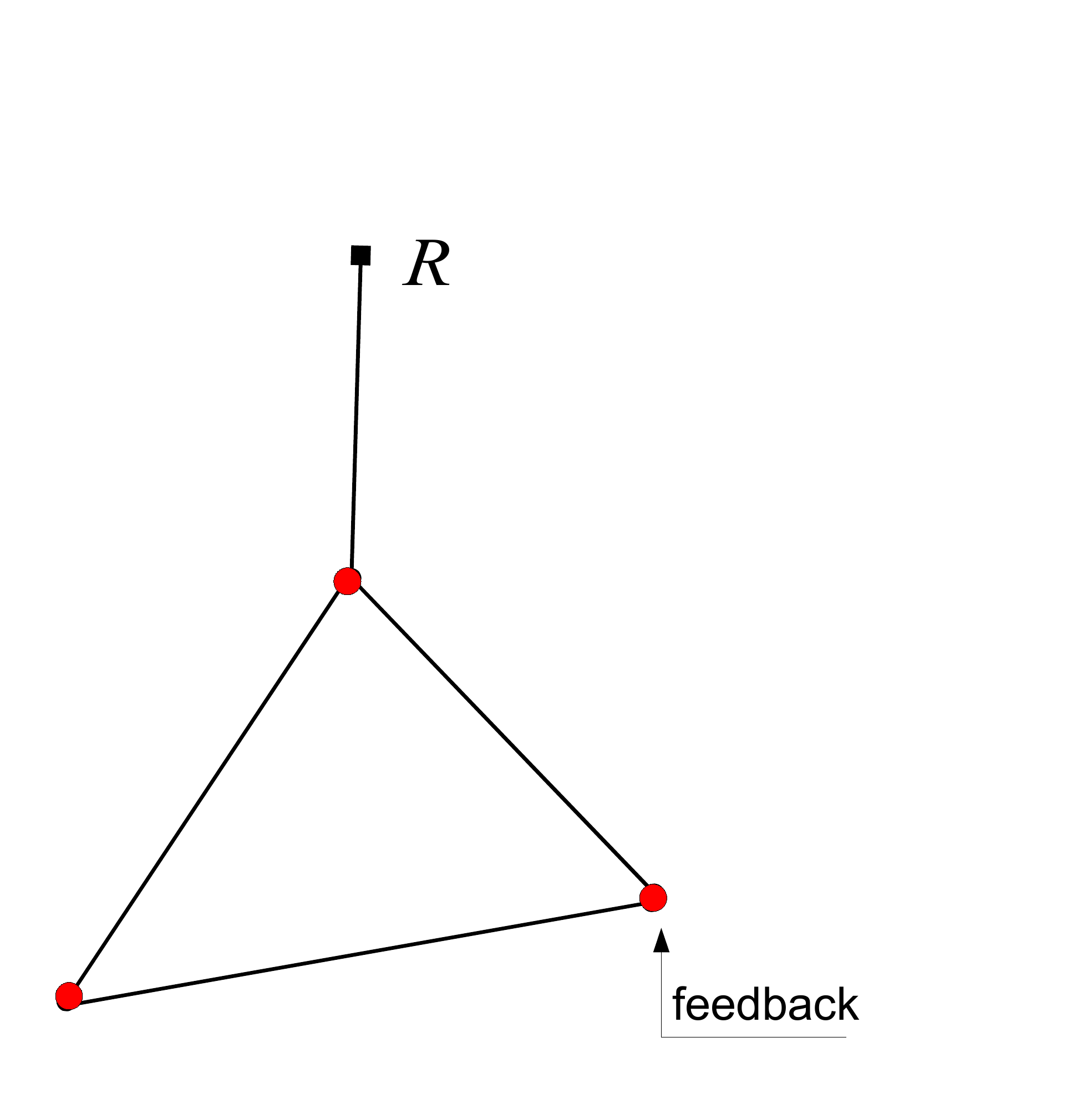}
\caption{Circuit}
\label{figure2}
\hfill
\end{minipage}
\begin{minipage}[t]{7cm}
\centering
\includegraphics[width=8cm, keepaspectratio =true]{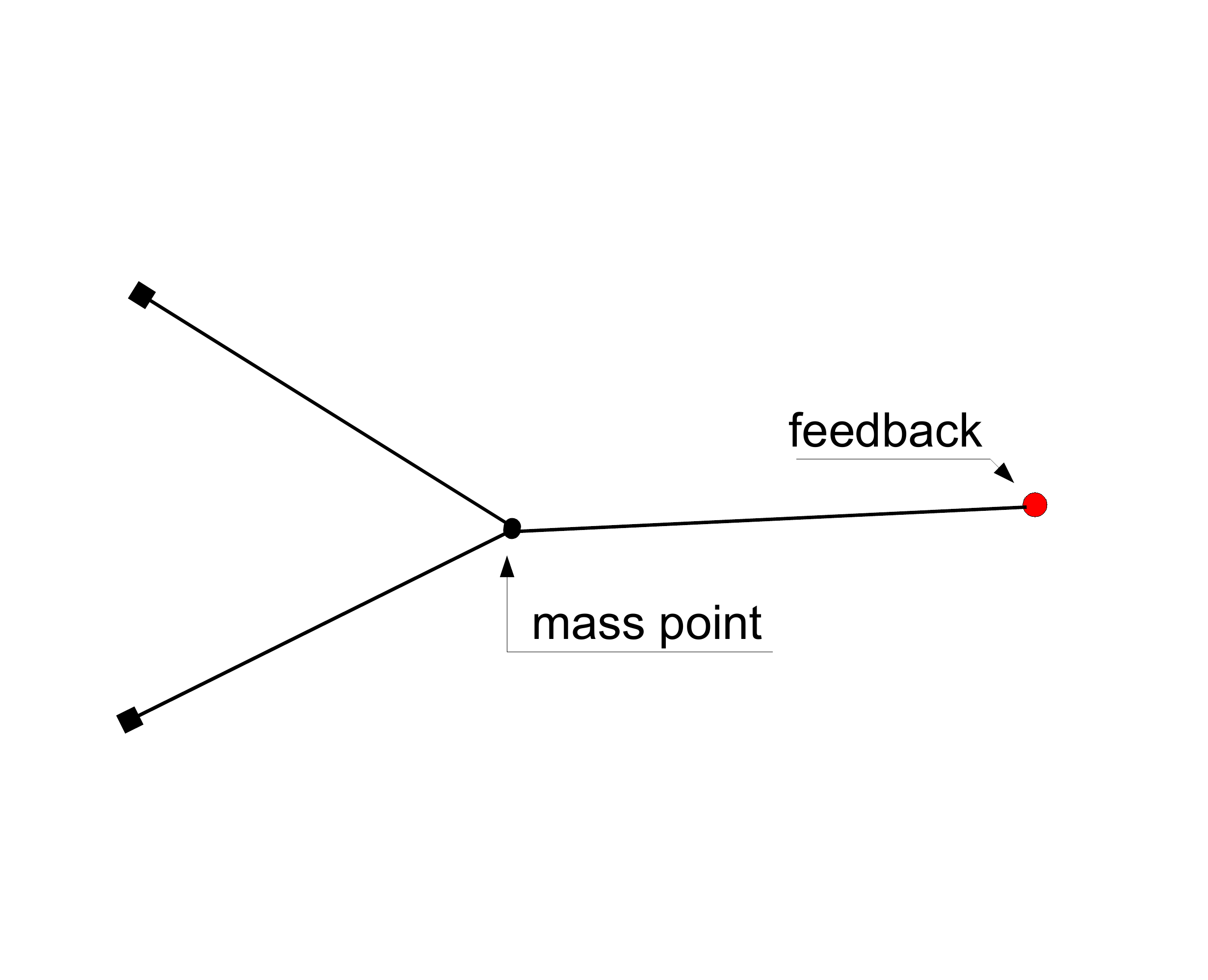}
\caption{Start}
\label{figure3}
\end{minipage}
\end{figure}

\subsection{A circuit}

(Figure \ref{figure2})

In this part we suppose that $\mathcal{T}$ contains a circuit (Figure \ref
{figure2} ), with feedbacks at each inner node. Then the second equation in (%
\ref{1.1}) will be 
\begin{equation*}
\sum\limits_{j\in J_{k}}d_{kj}\partial _{x}y^{j}(a_{k},t)=s_{1}^{\prime
}(t)-y_{t}(a_{k},t),\;\;a_{k}\in I_{\mathcal{M}}
\end{equation*}
and we can rewrite the system (\ref{1.1}) in the Hilbert space $H$ as 
\begin{equation*}
z^{\prime }(t)=\mathcal{A}z(t).
\end{equation*}
Except that, here, the first equation in (\ref{1.33}) will be replaced by
the following 
\begin{equation*}
\sum_{j\in J_{k}}d_{kj}\dfrac{dy^{j}}{dx}(a_{k})=q_{k}-y_{t}(a_{k}),\;%
\forall k\in I_{\mathcal{M}}.
\end{equation*}

The operator $\mathcal{A}$ generates a $\mathcal{C}_{0}$-semigroup of
contraction $(S(t))_{t\geq 0}$.

Next we suppose, without loss of generality, that $\ell _{1}=\ell _{2}=\ell
_{3}=1.$ We have the first result of asymptotic behavior

\begin{theorem}
$(S(t))_{t \leq 0}$ is asymptotically stable if and only if $\ell _{4}$ is irrational
and not in $\pi \mathbb{Z}.$
\end{theorem}

\begin{proof}
First, if $\ell _{4}$ is in $\pi \mathbb{Z}$ then $\mathbf{i}$ is an
eigenvalue of $\mathcal{A}$ (as in the case of a tree) and if $\ell _{4}=%
\frac{a}{b}$ with $a$ and $b$ integer, then $\mathbf{i}b\pi $ is an
eigenvalue of $\mathcal{A}$.

Now, we suppose that $\ell _{4}\notin \mathbb{Q}$ and $\ell _{4}\notin \pi 
\mathbb{Z}.$ We only need to prove that $\mathbf{i}\mathbb{R\subset \rho (%
\mathcal{A})}.$ For this, we will prove that the equation 
\begin{equation}
\mathcal{A}z=\mathbf{i}\beta z  \label{2.61}
\end{equation}
with $z=\left( 
\begin{array}{c}
y \\ 
v \\ 
p \\ 
q
\end{array}
\right) \in \mathcal{D}(\mathcal{A})$ and $\beta \in \mathbb{R}$ has only
trivial solution.

we can show that $v(a_{k})=0$ for $k\in I_{\mathcal{M}}.$ Now the equation (%
\ref{2.61}) can be rewritten explicitly as 
\begin{eqnarray}
v^{j} &=&\mathbf{i}\beta y^{j},\;\;j\in J,  \label{sss} \\
\dfrac{d^{2}y^{j}}{dx^{2}} &=&\mathbf{i}\beta v^{j},\;\;j\in J,  \notag \\
q_{k} &=&\mathbf{i}\beta p_{k},\;\;k\in I_{\mathcal{M}},  \label{qq} \\
-p_{k}-v(a_{k}) &=&\mathbf{i}\beta q_{\kappa },\;\;k\in I_{\mathcal{M}}. 
\notag
\end{eqnarray}

\noindent If $\beta =0$ then as in for the initial example we show that $%
y=0. $

Next, we suppose that $\beta \neq 0.$ We have $y(a_{k})=0$ for every $k\in
I_{\mathcal{M}}$ 
\begin{eqnarray}
\beta ^{2}y^{j}+\dfrac{d^{2}y^{j}}{dx^{2}} &=&0,\;\;j\in J,  \label{eq1} \\
(\beta ^{2}-1)p_{k} &=&0,\;\;k\in I_{\mathcal{M}}.  \label{eq22}
\end{eqnarray}
Then, since $y(a_{k})=0$ for every $k$ then $y^{j}$ is of the form 
\begin{equation*}
y^{j}=\alpha _{j}\sin (\beta x).
\end{equation*}
As in the case of a tree we consider two cases: $\beta ^{2}=1$ and $\beta
^{2}\neq 1$

\textit{First case:} $\beta ^{2}=1$. we have $\alpha _{1}=\alpha _{2}=\alpha
_{3}=0$ and since $\ell _{4}\notin \pi \mathbb{Z}$, $\alpha _{4}=0.$ Return
back to the balance conditions and (\ref{qq}), at inner nodes, one can
deduce that $q=p=0$ and hence $y=0.$

\textit{Second case: }$\beta ^{2}\neq 1.$ we have $p_{k}=q_{k}=0$ for every $%
k$ in $I_{\mathcal{M}}.$ Then $y=0.$
\end{proof}

\begin{theorem}
The semigroup $(S(t))_{t\geq 0}$ is not exponentially stable, even if $\ell
_{4}$ is irrational and not in $\pi \mathbb{Z}.$
\end{theorem}

\begin{proof}
To prove that $(S(t))_{t\geq 0}$ is not exponentially stable we consider the
sequence $f_{n}$ of vectors of $H$ defined by $f_{n}=(0,g_{n},0,0)$ where $%
g_{n}=(0,-\sin \beta _{n}x,0,0)$ and $\beta _{n}$ is a sequence of real
numbers satisfying $\beta _{n}\longmapsto +\infty $ and will be defined
later. We then prove that the sequence $z_{n}=(y_{n},v_{n},p_{n},q_{n})$ of
elements of $\mathcal{D}(\mathcal{A})$ such that 
\begin{equation*}
(\mathbf{i}\beta _{n}-\mathcal{A})z_{n}=f_{n}
\end{equation*}
is not bounded.

The sequence $y_{n}$ should satisfy 
\begin{equation*}
\left\{ 
\begin{tabular}{l}
$\beta _{n}^{2}y_{n}^{j}+\partial _{x}^{2}y_{n}^{j}=0,\;\;$for $j=1,3,4$ \\ 
$\beta _{n}^{2}y_{n}^{2}+\partial _{x}^{2}y_{n}^{2}=\sin \beta _{n}x.$%
\end{tabular}
\right.
\end{equation*}
Then, for $j=1,3,4$ there exists two complex numbers $a_{j}$ and $b_{j}$
such that 
\begin{equation*}
\left\{ 
\begin{tabular}{l}
$y_{n}^{j}=a_{j}\sin (\beta _{n}x)+b_{j}\cos (\beta _{n}x),$ \\ 
$\frac{dy_{n}^{j}}{dx}=-\beta _{n}b_{j}\sin (\beta _{n}x)+\beta
_{n}a_{j}\cos (\beta _{n}x)$%
\end{tabular}
\right.
\end{equation*}
and there exists $a_{2}$ and $b_{2}$ in $\mathbb{C}$ such that 
\begin{equation*}
\left\{ 
\begin{tabular}{l}
$
\begin{tabular}{l}
$y_{n}^{2}=a_{2}\sin (\beta _{n}x)+(-\frac{x}{2\beta _{n}}+b_{2})\cos (\beta
_{n}x),$ \\ 
$\frac{dy_{n}^{2}}{dx}=(\frac{x}{2}-\beta _{n}b_{2})\sin (\beta _{n}x)+(-%
\frac{1}{2\beta _{n}}+\beta _{n}a_{2})\cos (\beta _{n}x).$%
\end{tabular}
$%
\end{tabular}
\right.
\end{equation*}
The boundary and transmission conditions are expressed as follows 
\begin{equation*}
\left\{ 
\begin{tabular}{l}
$a_{1}\sin (\beta _{n})+b_{1}\cos (\beta _{n})=0$ \\ 
$b_{1}=b_{2}=b_{3}$ \\ 
$-\frac{1}{2\beta _{n}}+\beta _{n}a_{1}+\beta _{n}a_{2}+\beta
_{n}a_{3}=-i\beta _{n}b_{1}$ \\ 
$a_{2}\sin (\beta _{n})+(-\frac{1}{2\beta _{n}}+b_{2})\cos (\beta
_{n})=b_{4} $ \\ 
$\beta _{n}a_{4}-\left( (\frac{1}{2}-\beta _{n}b_{2})\sin (\beta _{n})+(-%
\frac{1}{2\beta _{n}}+\beta _{n}a_{2})\cos (\beta _{n})\right) =-i\beta
_{n}b_{4}$ \\ 
$a_{3}\sin (\beta _{n})+b_{1}\cos (\beta _{n})=a_{4}\sin (\beta _{n}\ell
_{4})+b_{4}\cos (\beta _{n}\ell _{4})$ \\ 
$-\beta _{n}b_{3}\sin (\beta _{n})+\beta _{n}a_{3}\cos (\beta _{n})-\beta
_{n}b_{4}\sin (\beta _{n}\ell _{4})+\beta _{n}a_{4}\cos (\beta _{n}\ell
_{4})=$ \\ 
$i\beta _{n}(a_{3}\sin (\beta _{n})+b_{1}\cos (\beta _{n})).$%
\end{tabular}
\right.
\end{equation*}
Then 
\begin{equation*}
\left\{ 
\begin{tabular}{l}
$a_{1}\sin (\beta _{n})+b_{1}\cos (\beta _{n})=0$ \\ 
$-\frac{1}{2\beta _{n}}+\beta _{n}a_{1}+\beta _{n}a_{2}+\beta
_{n}a_{3}=-i\beta _{n}b_{1}$ \\ 
$a_{2}\sin (\beta _{n})+(-\frac{1}{2\beta _{n}}+b_{1})\cos (\beta
_{n})=b_{4} $ \\ 
$\beta _{n}a_{4}-\left( (\frac{1}{2}-\beta _{n}b_{1})\sin (\beta _{n})+(-%
\frac{1}{2\beta _{n}}+\beta _{n}a_{2})\cos (\beta _{n})\right) =-i\beta
_{n}b_{4}$ \\ 
$a_{3}\sin (\beta _{n})+b_{1}\cos (\beta _{n})=a_{4}\sin (\beta _{n}\ell
_{4})+b_{4}\cos (\beta _{n}\ell _{4})$ \\ 
$\beta _{n}a_{3}e^{-\mathbf{i}\beta _{n}}-\beta _{n}b_{4}\sin (\beta
_{n}\ell _{4})+\beta _{n}a_{4}\cos (\beta _{n}\ell _{4})=i\beta _{n}b_{1}e^{-%
\mathbf{i}\beta _{n}}.$%
\end{tabular}
\right.
\end{equation*}
Hence 
\begin{equation}
(FB+AG)\beta _{n}b_{1}=AH-F\beta _{n}C  \label{eqcir}
\end{equation}
where 
\begin{eqnarray*}
A &=&\left( 1+\cos (\beta _{n}\ell _{4})\right) \sin (\beta _{n})+e^{-%
\mathbf{i}\beta _{n}}\sin (\beta _{n}\ell _{4}), \\
B &=&(2-\cos (\beta _{n}\ell _{4})\cos (\beta _{n})+\mathbf{i}\left( e^{-%
\mathbf{i}\beta _{n}}\sin (\beta _{n}\ell _{4})-\sin (\beta _{n})\right) , \\
C &=&\frac{1}{2\beta _{n}^{2}}\sin (\beta _{n})-(\frac{1}{2}\sin (\beta
_{n})+\frac{1}{2}(-1+\mathbf{i})\cos (\beta _{n}))\sin (\beta _{n}\ell _{4})
\\
&+&\frac{1}{2\beta _{n}}\cos (\beta _{n})\cos (\beta _{n}\ell _{4})
\end{eqnarray*}
and 
\begin{eqnarray*}
F &=&e^{-\mathbf{i}\beta _{n}}\left( \cos \beta _{n}-1\right) -\sin (\beta
_{n})\sin (\beta _{n}\ell _{4}), \\
G &=&e^{-\mathbf{i}\beta _{n}}\left( \cot \beta _{n}-2\mathbf{i}\right)
-\cos (\beta _{n})\sin (\beta _{n}\ell _{4})-\mathbf{i}e^{-\mathbf{i}\beta
_{n}}\cos (\beta _{n}\ell _{4}), \\
H &=&-\frac{1}{2\beta _{n}}e^{-\mathbf{i}\beta _{n}}-\frac{1}{2\beta _{n}}%
\cos (\beta _{n})\sin (\beta _{n}\ell _{4})-\frac{1}{2}\sin (\beta _{n})\cos
(\beta _{n}\ell _{4}) \\
&-&\frac{1}{2}(-1+\mathbf{i})\cos (\beta _{n})\cos (\beta _{n}\ell _{4}).
\end{eqnarray*}
Now by using the Asymptotic Dirichlet's theorem \cite{Shm80}, there exists $%
(p_{n},q_{n})\in \mathbb{N}^{2}$ such that $\frac{p_{n}}{q_{n}}$ converge to 
$\ell _{4},$ $p_{n}$ and $q_{n}$ tend to infinity as $n$ goes to infinity
and for every $n$ in $\mathbb{N}$ 
\begin{equation*}
\left| q_{n}\ell _{4}-p_{n}\right| <\frac{1}{q_{n}}.
\end{equation*}
Take $\beta _{n}=2\pi q_{n}+\frac{2\pi }{q_{n}^{1/4}},$ then there exists a
positive integer $n_{0}$ such that for every integer $n\geq n_{0},$ 
\begin{equation*}
0<\lambda _{n}:=-\frac{2\pi }{q_{n}}+\frac{2\pi \ell _{4}}{q_{n}^{1/4}}%
<\beta _{n}\ell _{4}-2\pi p_{n}<\mu _{n}:=\frac{2\pi }{q_{n}}+\frac{2\pi
\ell _{4}}{q_{n}^{1/4}}<\frac{\pi }{2}
\end{equation*}
and 
\begin{equation*}
\sin (\lambda _{n})<\sin (\beta _{n}\ell _{4})<\sin (\mu _{n}),\;\;\cos
(\lambda _{n})<\cos (\beta _{n}\ell _{4})<\cos (\mu _{n}).
\end{equation*}
Moreover $\sin (\beta _{n}\ell _{4})$ and $\cos (\beta _{n}\ell _{4})$
satisfy the following asymptotic approximations, 
\begin{equation*}
\sin (\beta _{n}\ell _{4})=\frac{2\pi \ell _{4}}{q_{n}^{1/4}}+o(\frac{1}{%
q_{n}^{1/4}}),\;\;\cos (\beta _{n}\ell _{4})=1+o(\frac{1}{q_{n}^{1/4}}).
\end{equation*}
We have also 
\begin{equation*}
\sin (\beta _{n})=\sin (\frac{2\pi }{q_{n}^{1/4}})=\frac{2\pi }{q_{n}^{1/4}}%
+o(\frac{1}{q_{n}^{1/4}}),\cos (\beta _{n})=1+o(\frac{1}{q_{n}^{1/4}}),\;%
\text{and }\cot (\beta _{n})=\frac{q_{n}^{1/4}}{2\pi }(1+o(\frac{1}{%
q_{n}^{1/4}})).
\end{equation*}
It follows that 
\begin{eqnarray*}
A &=&\frac{2\pi (2+\ell _{4})}{q_{n}^{1/4}}+o(\frac{1}{q_{n}^{1/4}}%
),\;\;\;B=1+o(1),\;\;\;\;C=(-1+\mathbf{i})\frac{\pi \ell _{4}}{q_{n}^{1/4}}%
+o(\frac{1}{q_{n}^{1/4}}), \\
F &=&-\frac{2\pi ^{2}(2\ell _{4}+1)}{q_{n}^{1/2}}+o(\frac{1}{q_{n}^{1/2}}%
),\;\;\;G=\frac{q_{n}^{1/4}}{2\pi }-4\mathbf{i}+o(1),\;\;\;H=-\frac{1}{2}(-1+%
\mathbf{i})+o(1).
\end{eqnarray*}
Return back to (\ref{eqcir}), we could write 
\begin{equation*}
(2+\ell _{4}+o(1))\beta _{n}b_{1}=o(1)-\beta _{n}\left( -(-1+\mathbf{i})%
\frac{2\pi ^{3}(2\ell _{4}+1)\ell _{4}}{q_{n}^{3/4}}+o(\frac{1}{q_{n}^{3/4}}%
)\right) \sim 2\ell _{4}(2\ell _{4}+1)(-1+\mathbf{i})\pi ^{3}q_{n}^{1/4}.
\end{equation*}
Hence 
\begin{equation*}
\beta _{n}b_{1}\sim 2\ell _{4}\frac{2\ell _{4}+1}{\ell _{4}+2}(-1+\mathbf{i}%
)\pi ^{3}q_{n}^{1/4}.
\end{equation*}

Which implies that $\left\| y_{n}^{1}\right\|$ converges to
infinity as $n$ goes to infinity and that consequently $z_n$ is not bounded.
\end{proof}

\begin{remark}
A small change in the proof leads to the conclusion that a polynomial
stability, can not be better than $\frac{1}{t^{2}}$ (by using a frequency
domain characterization of polynomial stability of a $\mathcal{C}_{0}$%
-semigroup of contraction due to Borichev and Tomilov \cite{TOM10})
Precisely we prove that the system is not $\frac{1}{t^{\alpha }}$%
-polynomially stable for every $\alpha $ in $(0,2).$
\end{remark}

\subsection{A start with two fixed endpoints}

(Figure \ref{figure3}) 
\begin{equation}
\left\{ 
\begin{tabular}{l}
$y_{tt}^{j}-y_{xx}^{j}=0$ in $(0,\ell _{j})\times (0,\infty ),\;\;j\in
\{1,2,3\},$ \\ 
\\ 
$\sum\limits_{j=1}^{3}y_{x}^{j}(0,t)=s^{\prime }(t),\;\;s^{\prime \prime
}(t)+s(t)=-y_{t}(0,t),$ \\ 
$y^{j}(0,t)=y^{l}(0,t),\;\;j,l\in \{1,2,3\},$ \\ 
$y^{2}(1,t)=y^{3}(1,t)=0,\;\;y_{x}^{1}(1,t)=-y_{t}(1,t),$ \\ 
\\ 
$s(0)=s_{0},\;\;s^{\prime }(0)=s_{1},$ \\ 
$y^{j}(x,0)=y_{0}^{j}(x),$ $\;y_{t}^{j}(x,0)=y_{1}^{j}(x),$ \ $x\in
(0,1),\;j\in \{1,2,3\}.$%
\end{tabular}
\right.  \label{3.3}
\end{equation}

We can rewrite the system (\ref{3.3}) in the Hilbert space $H$ as 
\begin{equation*}
z^{\prime }(t)=\mathcal{A}z(t),
\end{equation*}
where 
\begin{equation*}
H=V\times \prod_{j=1}^{3}L^{2}(0,\ell _{3})\times \mathbb{C}^{2}
\end{equation*}
with 
\begin{equation*}
V=\left\{ \Phi \in \prod_{j\in J}H^{1}(0,\ell _{j}),\;\Phi ^{2}(\ell
_{2})=\Phi ^{3}(\ell _{3})=0,\;\Phi ^{j}(0)=\Phi ^{l}(0),\;j,l\in
\{1,2,3\}\right\} .
\end{equation*}
and 
\begin{equation*}
\mathcal{A}\left( 
\begin{array}{c}
y \\ 
v \\ 
p \\ 
q
\end{array}
\right) =\left( 
\begin{array}{c}
v \\ 
\partial _{x}^{2}y \\ 
q \\ 
-p-v(0)
\end{array}
\right) ,\forall \left( 
\begin{array}{c}
y \\ 
v \\ 
p \\ 
q
\end{array}
\right) \in \mathcal{D}(\mathcal{A}),
\end{equation*}
with 
\begin{eqnarray*}
\mathcal{D}(\mathcal{A}) &=&\left\{ (y,v,p,q)\in \left[
\prod_{j=1}^{3}H^{2}(0,\ell _{j})\cap V\right] \times V\times \mathbb{C}%
^{2};\right. \\
&&\left. \sum_{j=1}^{3}\dfrac{dy^{j}}{dx}(0)=q,\;\text{ and }\dfrac{dy^{1}}{%
dx}(\ell _{1})=-v^{1}(\ell _{1})\right\} .
\end{eqnarray*}

The operator $\mathcal{A}$ generates a $\mathcal{C}_{0}$-semigroup of
contraction $(S(t))_{t\geq 0}$.

If we take $\ell _{1}=\ell _{2}=1$ then we have the following result.

\begin{theorem}
the semigroup $(S(t))_{t \geq 0}$ is asymptotically stable if and only if $%
\ell _{3}$ irrational and not in $\pi \mathbb{Z}.$\newline
Even if $\ell _{3}$ irrational and not in $\pi \mathbb{Z}$, the semigroup $%
(S(t))_{t \geq 0}$ is not exponentially stable.
\end{theorem}

\begin{proof}
As in the previous case, there exists $(p_{n},q_{n})\in \mathbb{N}^{2}$ such
that $\frac{p_{n}}{q_{n}}$ converge to $\ell _{4}$ as $n$ goes to infinity.
Then $q_{n}$ converge to infinity and there is a subsequence of $q_{n}$
denoted $q_{n}$ such that 
\begin{equation*}
\left| q_{n}\ell _{3}-p_{n}\right| <\frac{1}{q_{n}}
\end{equation*}
Then we take $\beta _{n}=2\pi q_{n}+\frac{2\pi }{q_{n}^{1/4}}$ and $%
f_{n}=(0,g_{n},0,0)$ with $g_{n}=(0,-\sin \beta _{n}x,0).$ The sequence $%
\beta _{n}$ tends to infinity as $n$ goes to infinity, $f_{n}\in H$ and $%
f_{n}$ is bounded. To conclude, we prove as in the previous case that $z_{n}$
defined by $(\mathcal{A}-\mathbf{i}\beta _{n})z_{n}=f_{n}$ is bounded.
\end{proof}

\section{A chain with non equal mass points} \label{sec5}

In this section, we consider a particular network which is a chain of $N$
edges ($N\geq 2$) and $p=N+1$ vertices such that the $(N-1)$ interior
vertices $a_{j}$ are point masses with mass $m_{j}.$ But the masses $m_{j}$
are not necessery equal (Figure \ref{figure_6}). 
\begin{figure}[th]
\centering
\includegraphics[width=18cm, keepaspectratio =true]{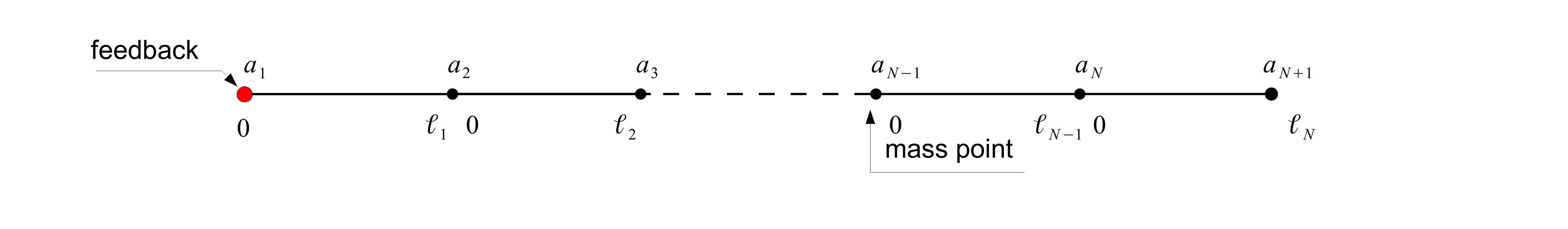}
\caption{A chain of strings}
\label{figure_6}
\end{figure}

Precisely, we consider the following system 
\begin{equation}
\left\{ 
\begin{tabular}{l}
$y_{tt}^{j}-y_{xx}^{j}=0$ in $(0,\ell _{j})\times (0,\infty ),\;\;j\in
\{1,...,N\}$, \\ 
\\ 
$y_{x}^{j}(0,t)-y_{x}^{j-1}(\ell _{j-1},t)=s_{j}^{\prime }(t),\;\;\;\;\;j\in
\{2,...,N\},$ \\ 
$m_{j}s_{j}^{\prime \prime }(t)+s_{j}(t)=-y_{t}(0,t),\;\;\;\;\;j\in
\{2,...,N\},$ \\ 
$y^{j-1}(\ell _{j-1},t)=y^{j}(0,t),\;\;\;\;j\in \{2,...,N\},$ \\ 
$y_{x}^{1}(0,t)=y_{t}(0,t),$ \\ 
$y^{N}(\ell _{N},t)=0,$ \\ 
\\ 
$s_{j}(0)=s_{j,0},\;\;s_{j}^{\prime }(0)=s_{j,1},\;\;\;\;\;j\in \{2,...,N\},$
\\ 
$y^{j}(x,0)=y_{0}^{j}(x),$ $\;y_{t}^{j}(x,0)=y_{k}^{1}(x),$ \ $x\in (0,\ell
_{j}),\;j\in \{1,...,N\}.$%
\end{tabular}
\right.  \label{chain}
\end{equation}

Note that the feedback is applied at the vertex $a_{1}.$ We give a necessary
and sufficient condition for the exponential stability of system (\ref{chain}%
). The general case of a tree with distinct masses at inner nodes is
complicated for the moment.

To start we quickly redefine the associated state space $H$ and the operator 
$\mathcal{A}$ as follow: 
\begin{equation*}
H=V\times \prod_{j=1}^{N}L^{2}(0,\ell _{j})\times \left( \prod_{j=1}^{N-1}%
\mathbb{C}\right) ^{2}
\end{equation*}
with 
\begin{equation*}
V=\left\{ \Phi \in \prod_{j=1}^{N}H^{1}(0,\ell _{j}),\;\Phi ^{N}(\ell
_{N})=0,\;\Phi ^{j-1}(\ell _{j-1})=\Phi ^{j}(0),\;j=2,...,N\right\} .
\end{equation*}
and 
\begin{equation*}
\mathcal{A}\left( 
\begin{array}{c}
y \\ 
v \\ 
p \\ 
q
\end{array}
\right) =\left( 
\begin{array}{c}
v \\ 
\dfrac{d^{2}y}{dx^{2}} \\ 
q \\ 
-m^{-1}(p+v_{\mathcal{M}})
\end{array}
\right) ,\forall \left( 
\begin{array}{c}
y \\ 
v \\ 
p \\ 
q
\end{array}
\right) \in \mathcal{D}(\mathcal{A}),
\end{equation*}
with $-m^{-1}(p+v_{\mathcal{M}})=(-\frac{1}{mj}p_{j}-\frac{1}{m_{j}}v(\ell
_{j}))_{j\in \{2,...,N-1\}}$ and 
\begin{equation*}
\mathcal{D}(\mathcal{A})=\left\{ (y,v,p,q)\in \left[ \prod_{j\in
J}H^{2}(0,\ell _{j})\cap V\right] \times V\times  \prod_{k\in I_{%
\mathcal{M}}}\mathbb{C}^2\text{ satisfying (\ref{1.334}) }\right\}
\end{equation*}
where 
\begin{equation}
\left\{ 
\begin{tabular}{l}
$\dfrac{dy^{j}}{dx}(0,t)-\dfrac{dy^{j-1}}{dx}(\ell
_{j-1})=q_{j},\;\;\;\;\;j\in \{2,...,N\},$ \\ 
$\dfrac{dy^{1}}{dx}(0)=v^{1}(0).$%
\end{tabular}
\right.  \label{1.334}
\end{equation}

Then the operator $\mathcal{A}$ generates a $\mathcal{C}_{0}$-semigroup of
contractions $(S(t))_{t\in \mathbb{R}^{+}}.$ Moreover, $\sigma (\mathcal{A}%
)=\sigma _{p}(\mathcal{A}).$

For every mass point $m$ we denote by $i_{1}(m),...,i_{k_{m}}(m)$ the
indices of the interior nodes with masses equal to $m$ and ordered as
follow, $i_{1}(m)<i_{2}(m)<...<i_{k_{m}}(m).$

For $r=r(m)\in \{1,...,k_{m}\}$ we define the scalars 
\begin{equation*}
\Pi_{m,r(m),s}=
\end{equation*}
\begin{equation*}
\sum_{i_{r}=j_{0}<j_{1}<\cdots <j_{s}<i_{r+1}}\left(
\prod\limits_{i=0}^{s-1}c_{j_{i+1}}\sin (\beta \ell _{j_{i}}+\cdots +\beta
\ell _{j_{i+1}-1})\right) \sin (\beta \ell _{j_{s}}+\cdots +\beta \ell
_{i_{r+1}-1})
\end{equation*}

and 
\begin{eqnarray*}
\Delta _{r(m)} &:&=(-1)^{i_{r+1}-i_{r}}\sin (\beta \ell _{i_{r}}+ ... +\beta
\ell _{i_{r+1}-1}) \\
&+& \sum_{s=0}^{i_{r+1}-i_{r}-1}(-1)^{i_{r+1}-i_{r}+s}\Pi_{m,r(m),s}, \\
\text{with,\ \ }\Delta _{r(m)} &=&\sin (\beta \ell _{r})\text{ if \ }k_{m}=1.
\end{eqnarray*}

Then we have the following result of asymptotic behavior of the system (\ref
{chain}):

\begin{lemma}
The system defined by (\ref{chain}) is exponentially stable if and only if
for every mass point $m$ and for every $r(m),$ $\Delta _{r(m)}\neq 0.$
\end{lemma}

\begin{proof}
The first question is whether $\mathbf{i}\mathbb{R}$ belongs to $\rho (A).$
Thus, we will solve the equation 
\begin{equation}
\mathcal{A}z=\mathbf{i}\beta z  \label{exp2"}
\end{equation}
of unknown $z=\left( 
\begin{array}{c}
y \\ 
v \\ 
p \\ 
q
\end{array}
\right) \in \mathcal{D}(\mathcal{A})-\{0\}$ and $\beta \in \mathbb{R}$.

If $\beta =0$ then $z=0.$ Thus suppose that $\beta \neq 0.$ We have 
\begin{eqnarray}
v^{j} &=&\mathbf{i}\beta y^{j},\;\;\beta ^{2}y^{j}+\dfrac{d^{2}y^{j}}{dx^{2}}%
=0,\;\;j\in \{1,....,N\},  \label{eq'} \\
q^{j} &=&\mathbf{i}\beta p^{j},\;\;(m_{j}\beta
^{2}-1)p_{j}=v(a_{j}),\;\;j\in \{2,....,N\}.  \label{eq2'}
\end{eqnarray}
The function $y^{j},$ $j\in \{1,....,N\}$ is then of the form 
\begin{equation*}
y^{j}=\alpha _{j}\cos (\beta x)+\gamma _{j}\sin (\beta x).
\end{equation*}
If $m_{j}\beta ^{2}\neq 1$ for every $j$ in $\{2,...,N\}$ then we prove by
iteration, starting with $j=1$, that $y^{j}=0$ for every $j$ in $\{2,...,N\}$
and consequently $z=0.$

Now we suppose the existence of a mass point $m$ with $m\beta ^{2}=1,$Let $%
i_{1}(m)<...<i_{k_{m}}(m)$ the indices of inner nodes with masses equal to $%
m.$ Then as in the first case $y^{j}=0$ for every $j<i_{1}$. Let $r=r(m)\in
\{1,...,k_{m}\}$ we have the following system 
\begin{equation}
\left\{ 
\begin{tabular}{l}
$y^{r}(0)=0,$ \\ 
for $j=i_{r}$ to $j=i_{r+1}-2,$ $y^{j}(\ell _{j})=y^{j+1}(0),$ \\ 
and $-\dfrac{dy^{j}}{dx}(\ell _{j})+\dfrac{dy^{j+1}}{dx}(0)=q^{j+1},$ \\ 
$y^{i_{r+1}-1}(\ell _{i_{r+1}-1})=0$%
\end{tabular}
\right.  \label{system}
\end{equation}
It is obvious that the system (\ref{exp2"}) has trivial solution if and only
if for every $m$ and every $r(m)$ the system (\ref{system}) has a trivial
solution.

By changes of indices we can suppose that $i_{r}=2,$ $i_{r+1}=N+1.$ The
system (\ref{system}) can be rewritten as 
\begin{equation*}
\left\{ 
\begin{tabular}{l}
$\gamma _{2}=0,$ \\ 
for $j=2$ to $j=N-1,$ $\alpha _{j}\cos (\beta \ell _{j})+\gamma _{j}\sin
(\beta \ell _{j})=\gamma _{j+1},$ \\ 
and $\alpha _{j}\sin (\beta \ell _{j})-\gamma _{j}\cos (\beta \ell _{j})+%
\frac{1}{\beta }\frac{1}{m_{j}-m}\alpha _{j+1}+\gamma _{j+1}=0,$ \\ 
$\alpha _{N}\cos (\beta \ell _{N})+\gamma \sin (\beta \ell _{N})=0$%
\end{tabular}
\right.
\end{equation*}
The matrix of such system is 
\begin{equation*}
S_{N}=\left( 
\begin{array}{cccccccccc}
1 & 0 & 0 &  &  &  &  &  &  &  \\ 
\cos x_{2} & \sin x_{2} & -1 & 0 &  &  &  &  &  &  \\ 
\sin x_{2} & -\cos x_{2} & c_{3} & 1 & 0 &  &  &  &  &  \\ 
0 & 0 & \cos x_{3} & \sin x_{3} & -1 & 0 &  & (0) &  &  \\ 
0 & 0 & \sin x_{3} & -\cos x_{3} & c_{4} & 1 & 0 &  &  &  \\ 
&  &  &  & \cdots & \cdots & \cdots & \cdots &  &  \\ 
&  &  &  & \cdots & \cdots & \cdots & \cdots & 0 &  \\ 
&  & (0) &  &  &  & \cos x_{N-1} & \sin x_{N-1} & -1 & 0 \\ 
&  &  &  &  &  & \sin x_{N-1} & -\cos x_{N-1} & c_{N} & 1 \\ 
&  &  &  &  &  & 0 & 0 & \cos x_{N} & \sin x_{N}
\end{array}
\right)
\end{equation*}
where $x_{j}=\beta \ell _{j}$ and $c_{j}=\frac{1}{\beta (m_{j}-m)}.$ We want
to calculate the determinant $\Delta _{N}$ of $S_{N}$ For this, Let $M_{N}$
the determinant of the matrix obtain from $S_{N}$ by replacing $\cos x_{N}$
and $\sin x_{N}$ in the last line by $\sin x_{N}$ and $-\cos x_{N}$
respectively.

One can verifies easily that 
\begin{eqnarray*}
\Delta _{2} &=&\sin x_{2},\;\;\Delta _{3}=-\sin (x_{2}+x_{3})+c_{3}\sin
x_{2}\sin x_{3}, \\
\Delta _{4} &=&\sin (x_{2}+x_{3}+x_{4})-c_{3}\sin x_{2}\sin
(x_{3}+x_{4})-c_{4}\sin (x_{2}+x_{3})\sin x_{4} \\
&+& c_{3}c_{4}\sin x_{2}\sin x_{3}\sin x_{4}, \\
M_{2} &=&-\cos x_{2},\;\;M_{3}=\cos (x_{2}+x_{3})-c_{3}\sin x_{2}\cos x_{3},
\\
M_{4} &=&-\cos (x_{2}+x_{3}+x_{4})+c_{3}\sin x_{2}\cos
(x_{3}+x_{4})+c_{4}\sin (x_{2}+x_{3})\cos x_{4} \\
&-& c_{3}c_{4}\sin x_{2}\sin x_{3}\cos x_{4}.
\end{eqnarray*}

We will prove by induction that for every integer $N\geq 2,$%
\begin{equation*}
\Delta _{N} =(-1)^{N}\sin
(x_{2}+...+x_{N})+(-1)^{N-1}\sum_{j=2}^{N-1}c_{j+1}\sin
(x_{2}+...+x_{j-1})\sin (x_{j}+...+x_{N}) 
\end{equation*}
\begin{equation}
+\sum_{s=2}^{N-2}(-1)^{N-1-s}\sum_{2=j_{0}<j_{1}<\cdots <j_{s}\leq N}\left(
\prod\limits_{i=0}^{s-1}c_{j_{i+1}}\sin (x_{j_{i}}+...+x_{j_{i+1}-1})\right)
\sin (\beta \ell _{j_{s}}+\cdots +\beta \ell _{N})  \label{rule1}
\end{equation}
and 
\begin{equation*}
M_{N} =(-1)^{N+1}\cos (x_{2}+...+x_{N})+(-1)^{N}\sum_{j=2}^{N-1}c_{j+1}\sin
(x_{2}+...+x_{j-1})\cos (x_{j}+...+x_{N})
\end{equation*}
\begin{equation}
+\sum_{s=2}^{N-2}(-1)^{N-s}\sum_{2=j_{0}<j_{1}<\cdots <j_{s}\leq N}\left(
\prod\limits_{i=0}^{s-1}c_{j+1}\sin (x_{j_{i}}+...+x_{j_{i+1}-1})\right)
\cos (x_{j_{s}}+...+x_{N}).  \label{rule2}
\end{equation}
Such rules are true for $N=2.$ and $N=3$ let $N\in \mathbb{N}$ with $N\geq 3$
and suppose that (\ref{rule1}) and (\ref{rule2}) are true. Some calculations
leads to 
\begin{eqnarray*}
\Delta _{N} &=&(-\cos x_{N}+c_{N}\sin x_{N})\Delta _{N-1}+(\sin
x_{N})M_{N-1}, \\
M_{N} &=&(-\sin x_{N}-c_{N}\cos x_{N})\Delta _{N-1}-(\cos x_{N})M_{N-1}
\end{eqnarray*}
We can now verify the rule (\ref{rule1}) to order $N:${\tiny 
\begin{eqnarray*}
&&\Delta _{N} \\
&=&\left( (-1)^{N}\sin
(x_{2}+...+x_{N})+(-1)^{N-1}\sum_{j=2}^{N-1}c_{j+1}\sin
(x_{2}+...+x_{j-1})\sin (x_{j}+...+x_{N})\right. \\
&&\left. +\sum_{s=2}^{N-3}(-1)^{N-1-s}\sum_{2=j_{0}\leq j_{1}<\cdots
<j_{s}\leq N-1}\left( \prod\limits_{i=0}^{s-1}c_{j_{i+1}}\sin
(x_{j_{i}}+...+x_{j_{i+1}-1})\right) \sin (x_{j_{s}}+...+x_{N})\right) \\
&&\left. +\sum_{s=2}^{N-3}(-1)^{N-2-s}\sum_{2=j_{0}\leq j_{1}<\cdots
<j_{s}\leq N-1}c_{N}\left( \prod\limits_{i=0}^{s-1}c_{j_{i+1}}\sin
(x_{j_{i}}+...+x_{j_{i+1}-1})\right) \sin (x_{j_{s}}+...+x_{N-1})\sin
x_{N}\right) \\
&=&\left( (-1)^{N}\sin
(x_{2}+...+x_{N})+(-1)^{N-1}\sum_{j=2}^{N-1}c_{j+1}\sin
(x_{2}+...+x_{j-1})\sin (x_{j}+...+x_{N})\right. \\
&&\left. +\sum_{s=2}^{N-2}(-1)^{N-1-s}\sum_{2=j_{0}\leq j_{1}<\cdots
<j_{s}\leq N}\left( \prod\limits_{i=0}^{s-1}c_{j_{i+1}}\sin
(x_{j_{i}}+...+x_{j_{i+1}-1})\right) \sin (x_{j_{s}}+...+x_{N})\right).
\end{eqnarray*}
} A similar calculus, using (\ref{rule2}) proves that (\ref{rule2}) is
verified in order $N.$

We can now state the following results:

The associated semigroup $S(t)$ is asymptotically stable if and only if $%
\Delta _{r(m)}$ is different from zero for every mass point $m$ and every $%
r(m).$ To conclude that $(S(t))_{t \geq 0}$ is exponentially stable, it
suffices to prove that (\ref{2.10}) is satisfied by $(S(t))_{t \geq 0},$
exactly as in the proof of Theorem \ref{th32}
\end{proof}

\section*{Acknowledgment} This work was carried out under Airbus Chair Grant. Financial support from Airbus Chair Grant is gratefully acknowledged.


\begin{thebibliography}{99}

\bibitem{Abd12} A.-B. Abdallah and F. Shel, Exponential Stability of a General Network of 1-$d$ Thermoelastic Materials, {\em Mathematical Control and Related Fields,} {\bf 2}, (2012), 1-16.
 
\bibitem{Amm01d} K. Ammari, A. Henrot and M. Tucsnak, Asymptotic Behaviour of the solutions and optimal location of the actuator for the pointwise stabilization of a string, {\em Asymptotic Analysis,} {\bf 28} (2001), 215-240.

\bibitem{Amm04} K. Ammari and M. Jellouli, Stabilization of star-shaped networks of strings, {\em Differential and Integral Equations,} {\bf 17} (2004), 1395-1410.

\bibitem{Amm05} K. Ammari, M. Jellouli and M. Khenissi, Stabilization of generic trees of strings, {\em J. Dynamical. Control. Systems,}, {\bf 11} (2005), 177-193.

\bibitem{Amm15} K. Ammari and S. Nicaise, {\em Stabilization of elastic systems by collocated feedback}, 2124, Springer,  Cham, 2015.
 
\bibitem{TOM10} A. Boritchev and Y. Tomilov, Optimal polynomial decay of functions and operator semigroups, {\em Math. Ann.}, {\bf 347} (2010), 455-478.
  
\bibitem{Bre83} H. Brezis, {\em Analyse Fonctionnelle, Th\'{e}orie et Applications}, Masson, Paris, 1983.
 
\bibitem{Zua06} R. D\`{a}ger and E. Zuazua, {\em Wave propagation, observation and control in 1-d
flexible multi-structures}, Springer-Verlag, Berlin, Math\'{e}matiques $\&$ Applications, 2006.

\bibitem{Erv14} S. Ervedoza and M. Vanninathan, Controllability of a simplified model of fluid-structure interaction, {\em COCV.,} {\bf 20} (2014), 547-575.
 
\bibitem{Hua85} F.-L. Huang, Characteristic condition for exponential stability of linear dynamical systems in {H}ilbert spaces, {\em Ann. Diff. Eqs.,} {\bf 1} (1985), 43-56.
 
\bibitem{Tuc13} Y. Liu, T. Takahashi and M. Tucsnak, Single input of a simplified fluid-structure interaction model, {\em ESAIM Control Optim. Calc. Var.,} {\bf 19} (2013), 20-42.

\bibitem{Mer08} D. Mercier and V. R\'{e}gnier, Spectrum of a Network of {E}uler-{B}ernoulli Beams, {\em J. Math. Anal. Appl.}, {\bf 342} (2008), 174-196.
  
\bibitem{Pru84} J. Pr\"{u}ss, On the spectrum of $\mathcal{C}_{0}$-semigroups, {\em Trans. Amer. Math. Soc.,} {\bf 284} (1984), 847-857.

\bibitem{Shm80} W.-M. Schmidt, {\em Diophantine approximation}, Lecture Notes in Mathematics, 785, Springer-Verlag, Berlin-Heidelberg-New York, 1980.

\bibitem{Far13} F. Shel, Exponential Stability of a Network of Beams, {\em J. Dynam. Control Syst.,} {\bf 21} (2015), 443-460.

\bibitem{Tuc09} M. Tucsnak and M. Vanninathan, Locally distributed control for a model of fluid-structure interaction, {\em System and Control Letters,} {\bf 58} (2009), 547-552.

\bibitem{Zua03} J.-L. Vazquez and E. Zuazua, Large time behavior for a simplified 1D model of fluid-solid interaction, {\em Comm. Partial Differential Equations,} {\bf 28} (2003), 1705-1738.

\bibitem{Zua06'} 
\bysame, Lack of collision in a simplified 1D model for fluid-solid interaction, {\em Math. Models Methods Appl. Sci.,} {\bf 16} (2006), 637-678.

\end{thebibliography}

\end{document}